\documentclass[11pt]{amsart}

\usepackage[utf8]{inputenc}
\usepackage[english]{babel}
\usepackage{amsthm}
\usepackage{amsmath}
\usepackage{amssymb}
\usepackage{enumerate}
\usepackage{cite}

\theoremstyle{plain}
\newtheorem{thm}{Theorem}
\newtheorem{lemma}{Lemma}

\newtheorem{cor}[lemma]{Corollary}

\theoremstyle{definition}

\theoremstyle{remark}

\newcommand{\fp}{\mathbb{F}_p}
\newcommand{\fq}{\mathbb{F}_q}
\newcommand{\F}{\mathbb{F}}
\newcommand{\f}{\pi}

\newcommand{\E}{{E}}

\newcommand{\ecq}{\E(\fq)}

\newcommand{\ffq}{\fq(\E)}

\DeclareMathOperator{\n}{N}

\newcommand{\Z}{\mathbb{Z}}
\renewcommand{\a}{\mathfrak{a}}
\renewcommand{\b}{\mathfrak{b}}

\newcommand{\p}{\mathfrak{p}}
\renewcommand{\P}{\mathfrak{P}}

\DeclareMathOperator{\End}{End}

\DeclareMathOperator{\im}{Im}
\DeclareMathOperator{\vol}{vol}
\newcommand{\Q}{\mathbb{Q}}

\renewcommand{\O}{\mathcal{O}}
\newcommand{\cc}{\mathfrak{c}}
\renewcommand{\l}{\mathfrak{l}}

\renewcommand{\t}{\vartheta}

\DeclareMathOperator{\ann}{ann}

\begin{document}

\title[On the elliptic curve endomorphism generator]{\bf On the elliptic curve endomorphism generator}
\author{L\'aszl\'o M\'erai}
\address{Johann Radon Institute for Computational and Applied Mathematics, Austrian Academy of Sciences, Altenbergerstr.\ 69, 4040 Linz, Austria}
\email{laszlo.merai@oeaw.ac.at}

\begin{abstract}
For an elliptic curve $\E$ over a finite field we define the point sequence $(P_n)$ recursively by $P_n=\t (P_{n-1})=\t^n(P_0)$ with an endomorphism $\t\in\End(\E)$ and with some initial point $P_0$ on $\E$. We study the distribution and the linear complexity of sequences obtained from $(P_n)$. 
\end{abstract}

\keywords{elliptic curves, complex multiplication, discrepancy, linear complexity,  power generator}

\subjclass[2000]{65C10, 14H52, 94A55, 94A60}

\maketitle

\let\thefootnote\relax\footnote{
The final publication is available at Springer via http://dx.doi.org/10.1007/s10623-017-0382-0
}

\section{Introduction}

For a prime power $q=p^k$ we denote by $\fq$ the field of $q$ elements. Let $\E$ be a nonsingular elliptic curve given by an affine Weierstrass equation
\[
 \E: \ y^2 + (a_1 x +a_3)y=x^3+a_2x^2+a_4x+a_6
\]
with some $a_1,\dots, a_6\in\fq$, see \cite{ecbook} for a general background.

A natural way to construct a point sequence on a curve is to define the points recursively, i.e. for an initial point $P_o\in\E$ and a map $\psi:\E\rightarrow\E$ one can define the point sequence $(P_n)$ by the law
\begin{equation}\label{eq:gen_const}
 P_n=\psi(P_{n-1})\quad \text{for } n\geq 1.
\end{equation}

This sequence and its pseudorandomness properties have been widely studied in the literature for certain choices of $\psi$, see the survey paper \cite{shparlinski-survei}. For example, if $\psi$ is the translation map $\psi(P)=P+Q$ for some $Q\in\E$, then \eqref{eq:gen_const} defines the so-called \emph{elliptic curve congruential generator}. This generator has been suggested by Hallgren \cite{hallgren} and later its pseudorandomness properties have been intensively studied, see e.g \cite{BD2002,ElMS2002,HessShp}.
If $\psi$ is the multiplication map $\psi(P)=eP$ for some positive integer $e$, the generator defined by  \eqref{eq:gen_const} is called \emph{elliptic curve power generator} introduced by Lange and Shparlinski \cite{LS2005} and further studied in \cite{ElMS2008,M2014,MW2015+}.

In this paper we study the pseudorandomness properties of the sequence defined by \eqref{eq:gen_const} in the case when $\psi$ is an arbitrary \emph{endomorphism} (or \emph{complex multiplication}) of the curve.

Using endomorphisms to generate point sequences on curves, from a different point of view, has been introduced by Lange and Shparlinski \cite{LS2005B} (see also \cite{LS2007,M2016+}). Their construction is motivated by the fact that on certain curves, endomorphisms speed up the scalar multiplication. Examples for such curves are the \emph{Koblitz} or \emph{subfield curves}, see \cite{smart,solinas}. Moreover, using complex multiplication instead of scalar multiplication in \eqref{eq:gen_const} also allows to generate sequences of larger period length. 

\bigskip

As usual, we denote by $\ecq$ the $\fq$-rational points of the curve $\E$. We use $\infty$ to denote the point at infinity, which is the neutral element of the group of points on $\E$.

Let $\ffq$ denote the function field of the curve  $\E$, and let $x(\cdot)$ and $y(\cdot)$ be the coordinate functions. Specially, for a given point $P\in\ecq$ with $P\neq \infty$ we have $P=(x(P),y(P))$. Let $\deg(f)$ denote the degree of the pole divisor of $f$. In particular, $\deg(x) = 2$ and $\deg(y) = 3$.

We recall the Hasse-Weil bound
\begin{equation}\label{eq:H-W}
| \#  \ecq -q-1|\leq 2 q^{1/2}, 
\end{equation}
where $\#  \ecq$ is the cardinality of $\ecq$.

We define the \emph{Frobenius endomorphism} $\f$ which acts on a  point $P=(x,y)\in\E$, $P\neq\infty$ as
\[
 \pi(P)=(x^q,y^q)
\]
and $\pi(\infty)=\infty$.
Clearly, $\pi$ fixes $\ecq$. The Frobenius endomorphism $\pi$ has a characteristic polynomial over $\Z$ of the following form
\[
 \psi_{\pi}(X)=X^2-t\cdot X+q,
\]
where $t=q+1-\#\ecq$ is the \emph{trace} of $\pi$. By \eqref{eq:H-W} we also have $|t|\leq 2 q^{1/2}$. The curve $\E$ is called \emph{supersingular} if $p\mid t$, otherwise \emph{non-supersingular} or \emph{ordinary}.

Let $\End(\E)$ be the \emph{endomorphism ring} of $\E$. If $\E$ is an ordinary elliptic curve, then all the endomorphisms are defined over $\fq$ and the endomorphism ring is isomorphic to an order in the imaginary quadratic field $K=\Q(\sqrt{t^2-4q})$. If $D_K$ is the discriminant of $K$, we may write
\begin{equation}\label{eq:D_K}
\pi=\frac{t+v\sqrt{D_K}}{2} \quad \text{with }  t^2-4q=v^2 D_K.
\end{equation}
Then 
\begin{equation*}
 \Z[\pi]\subseteq \End(\E) \subseteq \O_K,
\end{equation*}
where $\O_K$ is the maximal order of $K$ (its ring of integers).

The discriminant of $\End(\E)$ is of the form $D_{\E}=u^2D_K$, where the conductor $u$ divides $v$ and uniquely determines $\End(\E)$. Using Schoof's algorithm \cite{schoof} one can compute the trace $t$ of $\pi$ in polynomial time which determines the order $\Z[\pi]$ generated by $\pi$. Factoring the discriminant $t^2-4q$ of $\Z[\pi]$, one can get $\O_K$ in subexponential time by \eqref{eq:D_K}. Then there are just finitely many possibilities for $\End(\E)$. Kohel \cite{kohel} obtained a deterministic algorithm to compute the conductor $u$, and so to determine $\End(\E)$, in time $O(q^{1/3+\varepsilon})$, assuming the generalized Riemann hypothesis (see also \cite{Bisson,BissonSuherland}).

\bigskip

For a given point $ P \in \E(\F_{q})$, $P\neq\infty$ and an endomorphism $\t\in\End(\E)$ define a point sequence $(P_n)$ by the rule
\begin{equation}\label{eq:seq}
 P_n=\t P_{n-1}=\t^n P, \quad n=1,2,\dots
\end{equation}
with the initial point $P_0=P$.

Clearly, the sequence $(P_n)$ is ultimately periodic. Let $\ell\geq 1$ be the order of $P$ and put: 
\[
\l=\{\alpha \in \End (E): \alpha P=\infty\}.  
\]
By definition, the order $\ell$ is the least positive integer in $\l$ so $\sqrt{\n(\l)}\leq \ell \leq \n(\l)$. We have $P_n=P_m$ iff $\t^n\equiv \t^m \mod \l$ and $P_n$ is purely periodic iff $\t$ is prime to $\l$, where the period length $T$ is the (multiplicative) order of $\t$ modulo $\l$.
If $\t$ is a scalar multiplication, then $T\leq \ell-1$. On the other hand, if $\t$ (and $P$) is chosen in a proper way one can achieve the period length $T=\ell^{2}-1$, namely when $\ell$ is prime in $\End(\E)$, so $\l=(\ell)$, and $\t$ is a primitive root modulo $\l$.

In this paper we first study exponential sums with sequence elements \eqref{eq:seq}. 
For a non-trivial additive character $\psi$ of $\fq$ and $f\in\F_{q}(\E)$ write
\[
 S_\t(\E,P,T)=\sum_{n=1}^T\psi\left(f\left(\t^nP\right)  \right).
\]
We prove a bound on $S_\t(\E,P,T)$. We use this result to study the distribution of the sequence  $(f(\t^nP))$.

We also investigate the linear complexity of the sequence $(f(\t^nP))$. We recall that the \emph{linear complexity} of a sequence $(s_n)$ over a field $\F$, is the length of a shortest linear recurrence relation
\begin{equation*}
 s_{n+L}=c_{L-1}s_{n+L-1}+\dots +c_1s_{n+1}+c_0s_n, \quad n=0,1,\dots
\end{equation*}
for some $c_0,\dots, c_{L-1}\in \F$, that $(s_n)$  satisfies.

The linear complexity measures the unpredictability of a sequence and thus its suitability in cryptography. 
For more details see \cite{MeidlWinterhof,Niederreiter2003,Winterhof2010}.

We give a lower bound on the linear complexity of the sequence $(f(\t^nP))$.

We also consider these questions in the special case when the discriminant $D_\E$ of the endomorphism ring $\End(\E)$ is small (i.e. both the CM-discriminant $D_K$ and the conductor $u$ are small). For a typical curve $\E$, the discriminant $D_\E$ tends to be large (see e.g. \cite{LucaShp2007}). However, for certain classes of curves, like the Koblitz or subfield curves \cite{smart,solinas}, the discriminant $D_\E$ is small. Moreover, ordinary pairing-friendly curves also need to have small CM-discriminant $D_K$ and usually they have small conductor $u$, see \cite{taxonomi}. We give sharper results in this case.

In the special case, when $\t\in\Z$, i.e. when $\t$ is a scalar multiplication instead of complex multiplication,  the period length $T$ becomes small, i.e.,  $T\leq \ell$, compared to the best period length $\ell^{2+o(1)}$. Consequently, our results become trivial. In this case one may need to use the earlier results of Lange and Shparlinski \cite{LS2005},  M\'erai \cite{M2014} and M\'erai and Winterhof \cite{MW2015+}.

\bigskip

We recall that $U\ll V$ is equivalent to the inequality $|U|\leq c V$ with some constant $c>0$.
Throughout the paper, the implied constants in ``$\ll$'' may sometimes, where obvious, depend on an integer parameter $\nu\geq 1$ and a real parameter $\varepsilon>0$ and are absolute otherwise.

\section{Preparation}
First we analyze the structure of the CM-torsion points of ordinary elliptic curves. We use these results to obtain bounds on character sums involving complex multiplications. Finally, we also prove some auxiliary results needed in the rest of the paper.

\subsection{CM-torsion points}

For an ideal $\a\vartriangleleft \End(\E)$ write
\[
 \E[\a]=\{P\in\E(\overline{\fq}): \ \alpha P=\infty \ \forall \alpha \in \a\}
\]
and put $\E[\alpha]=\E[(\alpha)]$ for $\alpha\in\End(\E)$. Clearly $\E[\a]$ forms a subgroup in $\E(\overline{\fq})$. 
It is well-known, that for an integer $a$ with $\gcd(a,q)=1$ we have
\begin{equation}\label{eq:a-torsion-1}
 \E[a]\cong \mathbb{Z}_a \times \mathbb{Z}_a.
\end{equation}
On the other hand, if  $a=p^{\nu}a'$ with $p \nmid a'$, then either
\begin{equation}\label{eq:a-torsion-2}
  \E[a]\cong \mathbb{Z}_{a} \times \mathbb{Z}_{a'} \quad  \text{or}  \quad \E[a]\cong \mathbb{Z}_{a'} \times \mathbb{Z}_{a'}.
\end{equation}
The first case occurs if and only if the curve is ordinary. Specially, we have that for an ordinary elliptic curve the cardinality of the $a$-torsion points is $\#\E[a]=a^2/p^{\nu}$.

Our first goal in this section is to investigate $\E[\a]$ for arbitrary $\a$.

We recall some basic facts about the number theoretic properties of orders of imaginary quadratic fields (for more details see \cite{cox,marcus}). Usually if the order $\O$ is not a maximal order in $\O_K$, then $\O$ is not UFD. However, if we restrict our-self to ideals which are prime to the conductor of the order, the unique factorization holds. 

\begin{lemma}\label{lemma:UFD}
Let $\O$ be an order of an imaginary quadratic field $K$ with conductor $u$. 
\begin{enumerate}
 \item An ideal $\a \vartriangleleft \O$ is prime to $u$ if and only if $\n(\a)$ is prime to $u$. 
 \item Let $I(\O,u)$ be the group of fractional ideals which are prime to $u$. Then $I(\O,u)\cong I(\O_K,u)$, and the isomorphism is given by $\a\mapsto \a\cap \O$ and its inverse is $\a \mapsto \a \O_K$.
\item If $\a$ is a fractional ideal of $\O$ such that $\n(\a)$ is prime to $u$ or $\a$ is principal, then $\n(\a)=\n(\a \O_K)$.
\item If $\a$ is an ideal, prime to $u$, then $\a$ is invertible, i.e. there is a fraction ideal $\a^{-1}$ such that $\a\a^{-1}=\O$.
\end{enumerate}
\end{lemma}

Throughout the section we will frequently use the M\"obius function $\mu_{K}$, the Euler's totient function $\varphi_K$ and the number of prime ideal divisors $\omega_K$ in the number field $K$. For rational integers it is well-known (see e.g. \cite{HardyWright}) that
\[
 \omega(k)\ll\frac{\log k}{\log \log (k+2)} \quad \text{and} \quad \varphi(k) \gg \frac{k}{\log \log (k+2)}.
\]
It also implies similar bounds on $\omega_{K}$ and $\varphi_K$. Let $\a\lhd \O_K$ be an arbitrary ideal.
As each rational prime has at most two prime ideal divisors, by $\a \mid \n(\a)$, it follows
\[
 \omega_{K}(\a)=\sum_{\p \mid \a} 1\leq 2\sum_{\substack{t\mid \n(\a)\\ t\in\mathbb{Z} \text{ is prime}}} 1\leq 2\,\omega(\n(\a)).
\]
Similarly,
\[
 \frac{\varphi_K(\a)}{\n(\a)}=\prod_{\p\mid \a}\left(1-\frac{1}{\n(\p)}\right)\geq \prod_{\substack{t \mid \n(\a)\\ t\in\mathbb{Z} \text{ is prime}}} \left(1-\frac{1}{t}\right)^2=\left(\frac{\varphi (\n(\a))}{\n(\a)}\right)^2.
\]
Thus
\[
 \omega_{K}(\a)\ll\frac{\log \n(\a)}{\log \log (\n(\a)+2)} \quad \text{and} \quad \varphi_K(\a) \gg \frac{\n(\a)}{(\log \log (\n(\a)+2))^2}.
\]

In the following we describe the subgroup  $\E[\a]$ of $\E(\overline{\fq})$ if $\a$ is prime to the conductor $u$. For scalar multiplication the group structure of the torsion points $\E[a]$ depends on whether the multiplication-by-$a$ map is separable (i.e. $a$ and $p$ are coprime) or not. A similar phenomenon occurs for complex multiplication. The non-separable endomorphisms form an ideal $\P$ generated by $\P=(p, \pi)$, see \cite[III. Corollary 5.5]{ecbook}. As $\n(\P)=p$, $\P$ is prime.

\begin{lemma}\label{lemma:torsion}
 Let $\E$ be an ordinary elliptic curve and let $u$ be the conductor of its endomorphism ring $\End(\E)$. Let $\a \vartriangleleft \End(\E)$ be an ideal such that it is prime to $u$ and not a power of $\P$. Write $\a=\b\cdot \P^\nu$ with $\b+\P=\End(\E)$. Then $\E[\a]\cong \End(\E)/\b$ as $\Z$-modules. Specially, $\#\E[\a]=\n(\a)/p^\nu$
\end{lemma}

\begin{proof}
First we prove the lemma for principal ideals generated by rational integers $a\in\Z$. Since each endomorphism $\alpha\in\End(\E)$ fixes $\E[a]$, $\alpha$ is also an endomorphism of $\E[a]$.  Moreover, $\alpha$ and $\beta$ induce the same map in $\E[a]$ if $\E[a]\leq \E[\alpha-\beta]$. For arbitrary $\a\lhd \End(\E)$ put
$I(\E[\a])=\{\sigma\in\End(\E) : \E[\a]\leq \E[\sigma]\}$. Clearly, $\a \leq I(\E[\a])$. Now,
$\End(\E[a])\geq \End(\E)/ I(\E[a])$. 
If $\gcd(a,p)=1$ (i.e. the multiplication-by-$a$ map is separable), we have $a\mid \sigma$ whenever $\sigma\in I(\E[a])$ (see e.g. \cite[III. Corollary~4.11]{ecbook}), thus $I(\E[a])=(a)$. Then
\[
\# \End(\E[a]) \geq  \# \End(\E)/(a)=  \n(a)=a^2.
\]
On the other hand, since $\E[a]$ is a finite abelian group, $\E[a]\cong\End (\E[a])$ as $\Z$-modules. Thus $\#\End(\E[a])=a^2$, and $\End(\E[a])\cong \End(\E)/(a)$. 

If $\gcd(a,p)\neq 1$, we get the assertion by the Chinese Remainder Theorem. Namely, if $a=a' \cdot p^\nu$ with $\gcd(a',p)=1$, write $\E[a]\cong \Z_{a'}\times \Z_{a'}\times\Z_{p^\nu}$. Then all endomorphisms can be represented by $e p^{\nu} \cdot \alpha +f a' \cdot \beta$ with $\alpha \in \End (\E[a'])$, $\beta\in\Z$ and the pair $(e,f)$ is a solution of $e  p^{\nu} +f  a'=1$.

Now let $\a\vartriangleleft \End(\E)$ be arbitrary and write $\n(\a)=a'\cdot p^{\nu}$. As $\E[\a]\leq \E[\n(\a)]$, we have $\E[\a]\cong \Z_{a_1}\times \Z_{a_2}$ for some positive integers $a_1,a_2$ with $a_1\mid a'$ and $a_2\mid a' p^{\nu}$. Consequently, each endomorphism of $\E[\a]$ can be represented by an endomorphism of $\E[\n(\a)]$, thus, by an endomorphism of $\E$. Moreover, 
\begin{equation}\label{eq:end-lower}
\End (\E[\a])\cong \End(\E)/I(\E[\a])\leq \End(\E)/\a.
\end{equation}

Let $\a=\b \P^{\nu}$ with $\b+\P=\End(\E)$, then 
\begin{equation}\label{eq:invariant}
\E[\a]=\E[\b\P^\nu]=\E[\b].
\end{equation}
Indeed, the inclusion $\E[\a] \supset \E[\b]$ is straightforward. To prove the other way inclusion, let $P\in\E[\a]$ be an arbitrary point. 
For all $\beta\in\b$, we have $\beta \cdot \pi^{\nu}\in\b\cdot \P=\a$, thus $\beta \pi^\nu \cdot  P=\infty$, so $\beta P=\infty$, i.e., $P\in\E[\b]$.  

Now we show 
\begin{equation}\label{eq:decomposition}
\E[\a]\oplus \E[\bar{\a}]=\E[\n(\a)]. 
\end{equation}
Let $a$ be the largest integer with $a\mid \a$, then $\a +\bar{\a}=a\End(\E)$. Write $a=\sigma_1\alpha_1 +\sigma_2\alpha_2 +\sigma_3\overline{\alpha_1} +\sigma_4\overline{\alpha_2}$ with $\alpha_1,\alpha_2\in\a$ and $\sigma_1,\dots, \sigma_4\in\End(\E)$. For a fixed $P\in\E[\n(\a)]$ let $Q\in \E$ such that $aQ=P$ (such a $Q$ exists, see e.g. \cite[III. Theorem 4.10]{ecbook}). Then $\sigma_1 (\alpha_1 Q)+\sigma_2 (\alpha_2 Q)+\sigma_3 (\overline{\alpha_1}Q)+\sigma_4 (\overline{\alpha_3}Q)=aQ=P$, where $\alpha_i Q\in\E[\a]$ and $\overline{\alpha_i} Q\in\E[\bar{\a}]$  for $i=1,2$  as $a\mid \alpha_1, \alpha_2$, which proves \eqref{eq:decomposition}.

To finish the proof put $\a=\cc \, \P^{\nu} \, \bar{\P}^{\mu}$ with $\cc+\P=\cc+\bar{\P}=\End(\E)$. Then $\n(\a)=\n(\cc)p^{\nu+\mu}$ with $p\nmid \n(\cc)$. By \eqref{eq:invariant}, we may suppose that $\nu=0$. Then by \eqref{eq:decomposition} we have
\[
\End (\E[\n(\a)])\cong \End (\E[\a]) \times \End (\E[\bar{\a}])=\End (\E[\a]) \times \End (\E[\bar{\cc}]).  
\]
Thus, $\#\End (\E[\n(\a)]) = \#\End (\E[\a]) \cdot \#\End (\E[\bar{\cc}])\leq \n(\a) \cdot \n(\cc)=\n(\a)^2/p^{\mu}$. Since $\#\End (\E[\n(\a)])=\n(\cc)^2p^{\mu}$ we get equality in \eqref{eq:end-lower} and the result follows.
\qed
\end{proof}

If $\E$ is an ordinary elliptic curve, then there are points with order $a$ which follows from \eqref{eq:a-torsion-1} and \eqref{eq:a-torsion-2} by the inclusion-exclusion principle. We prove the analogue result for complex multiplication. As $\E[\a]=\E[\a \P]$ for all ideals $\a$, we  assume that $\a$ is prime to $\P$.

If $P\in\E(\overline{\F_q})$ is a point, we define the \emph{annihilator} $\ann(P)$ of $P$ by
\[
 \ann(P)=\{\alpha\in\End(\E): \alpha P=\infty\}.
\]
Clearly, $\ann(P)$ is an ideal of $\End(\E)$.

\begin{lemma}\label{lemma:ann}
 Let $\E$ be an ordinary elliptic curve, and let $\a \vartriangleleft \End(\E)$ such that it is prime to $u$ and $\P$. Then there exists a point $P\in\E(\overline{\F_q})$ such that $\ann(P)=\a$.
\end{lemma}

\begin{proof} 
By Lemma \ref{lemma:UFD}, $\a$ has unique factorization. Then by Lemma \ref{lemma:torsion} and by the inclusion-exclusion principle we have that the number of points $P\in\E[\a]$ such that $P\not\in\E[\b]$ for any $\a \subsetneq \b$ is
 \[
 \sum_{ \b \cdot \cc =\a }  \mu_{\E}(\b) \n(\cc)=\n(\a) \sum_{\b \mid \a } \frac{\mu_{\E}(\b)}{\n(\b)}= \n(\a)\prod_{\substack{\p \mid \a  \\ \p \text{ is prime}}} \left(1-\frac{1}{\n(\p)} \right)>0,
 \]
where $\mu_{\E}$ is the M\"obius function in $I(\End(\E),u)$: $\mu_{\E}(\a)=\mu_K(\a\O_K)$. 
\qed
\end{proof}

\begin{lemma}\label{lemma:preimages}
Let $\t\in\End(\E)$ be an element prime to $u$ and $\P$ and let $Q\in\E(\overline{\F_q})$. Then there is a point $\overline{Q}\in \E(\overline{\F_q})$ such that $\t\overline{Q}=Q$ and $\ann(\overline{Q})=\t\ann(Q)$.
\end{lemma}

\begin{proof}
 First we prove that if $R,S\in\E(\overline{\F_q})$ two points with
 \begin{equation}\label{eq:preimages:cond}
  \ann(R)\subset \p^k, \quad \ann(S)\not\subset \p^k \quad \text{with } \p+(u)=\End(\E)
  \end{equation}
  for a prime ideal $\p$, then
 \begin{equation}\label{eq:preimages:claim} 
 \ann(R+S)\subset \p^k.
\end{equation}

Let $k$ be the maximal integer such that \eqref{eq:preimages:cond} holds and let $\b,\mathfrak{c}$ be ideals prime to $\p$ such that
\begin{equation}\label{eq:preimages-factor}
 \b\p^k\subset \ann(R)\subset \p^k, \quad  \mathfrak{c}\p^\ell\subset \ann(S)\subset \p^\ell
\end{equation}
for some $\ell<k$. 

We remark that such ideals exist. Indeed, assume that 
\begin{equation}\label{eq:preimages-ass}
\ann(R)\subset \p^k.                                                                                        
\end{equation}
As $\p$ is prime to $(u)$, it is invertible, i.e. there is a fractional ideal $\p^{-1}$ such that $\p^{-1}\p=\End(\E)$. Then by \eqref{eq:preimages-ass} and the maximality of $k$ we have 
\[
\p^{-k}\ann(R)\subset \End(\E) \quad \text{and } \p^{-k}\ann(R)\not\subset \p.
\]
Thus there exists an $\alpha\in \p^{-k}\ann(R)\setminus \p$. Then 
$\alpha \p^k\subset \ann(R)$, thus choosing $\b$ to be $(\alpha)$ we get the assertion~\eqref{eq:preimages-factor} for $\ann(R)$. One can prove the existence of $\mathfrak{c}$ in a similar way.

Now assume that \eqref{eq:preimages:claim} does not hold, i.e. there exists an $\alpha\in\ann(R+S)\setminus\p^k$. If $k_1$ is the maximal integer such that $\alpha\in\p^{k_1}$, consider the element $\delta=\alpha \gamma \tau^{\max\{\ell-k_1,0\}}$, $\delta\not\in\p^k$ for some $\gamma\in\mathfrak{c}\setminus \p$ and $\tau\in\p\setminus\p^2$. Then $\delta\in\ann(R+S)\cap\ann(S)$, thus $\delta\in\ann(Q)$ which contradicts $\delta\not\in\p^k$.

In order to prove the lemma let $S\in\E(\overline{\F_q})$ such that $\t S=Q$ (such a point exists, see e.g. \cite[III. Theorem 4.10]{ecbook}). We can assume that $\ann(S)\subset (\t)$. Indeed, if $\p^k$ is an ideal with $(\t)\subset \p^{k}$ but $\ann(S)\not\subset \p^k$, then by Lemma~\ref{lemma:ann} there is a point $R$ such that $\ann(R)=\p^k$, thus by \eqref{eq:preimages:claim} $\ann(R+S)\subset \p^k$, and $\t(R+S)=\t S=Q$. 

As $\ann(S)\subset (\t)$, each $\alpha\in\ann(S)$ can be written in the form $\alpha=\t\gamma$. As $\alpha S= \t\gamma S=\gamma Q=\infty$, thus $\gamma\in\ann(Q)$, i.e. $\ann(S)=\t\ann(Q)$.
\qed
\end{proof}

As an application, we show that certain linear combinations of images of a point with respect to different endomorphisms are nontrivial. We recall that two elements $\alpha, \beta \in\End(\E)$ are said to be associated elements if $\alpha/\beta$ is a unit in $K$.

\begin{lemma}\label{lemma:ratFn-nonconstant}
Let $\t_1,\dots, \t_s\in\End(\E)$ be pairwise non-associated elements with $(\t_i,u)=\End(\E)$, such that $0<\n(\t_1), \dots, \n(\t_s)\leq J$ and  $\t_i/\t_j$ ($i\neq j$) are not  powers of $\pi$.
For fixed elements $c_1\dots, c_s\in\fq$, not all of them are zero, and $f\in \ffq$
define the function
\[
 F(Q)=\sum_{i=1}^sc_i f(\t_i Q)\in\ffq.
\]
If $f$ has not the form $z^p-z$ with $z\in\overline{\ffq}$,
then the function $F$ has also not the form $z^p-z$ with $z\in\overline{\ffq}$, and it has degree $\deg F\leq s\, \deg f \, J$.
\end{lemma}
\begin{proof}
We may assume, that all the coefficients $c_1,\dots, c_s$ are non-zero. We may also assume, that $s\geq 2$.

Since $\deg f(\t_i Q)=\deg f \cdot \#\E[\t_i]$ we immediately get the assertion for the degree.

By assumption, $\E[\t_i]$ ($i=1,\dots, s$) are pairwise distinct.  Assume that $0\leq \#\E[\t_1]\leq \dots\leq \#\E[\t_s]\leq J$ with $\#\E[\t_s]\neq 0$. 

Let $R\in\E(\overline{\fq})$ be a pole of $f$ with minimal annihilator $\ann(R)$ and let $S\in\E(\overline{\F_q})$ a point such that $\t_sS=R$ and $\ann(S)=\t_s\ann(R)$, such a point exists by Lemma~\ref{lemma:preimages}. 

Assume, that there is another pole $S'$ of $F$ such that $S=S'$. Suppose, that $\t_iS'=R'$ for some $i$ and some pole $R'$ of $f$. As $\t_i\ann(R')\subset \ann(R')$, we get
\[
 \t_i\ann(R')\subset \ann(S')=\ann(S)=\t_s\ann(R).
\]
As $\ann(R)$ and $(\t_s)$ is minimal (for inclusion), we get $\t_i=\t_s$, so $R=\t_sS$ and $R'=\t_sS'$ agree.

Thus multiplicity of all poles $S$ of $F$ has the same multiplicity as the pole $R$ of $f$, where $R=\t_iS$ for some $i$. \qed
\end{proof}

Combining this result with the following  lemma \cite[Corollary 1]{KohelShparlinski} (see also~\cite{BD2002}) we can obtain a character sum estimate involving complex multiplication.

We recall that we define $\psi(f(Q))$ to be 0 whenever $Q$ is a pole of $f$.

\begin{lemma}
Let $\E$ be an ordinary elliptic curve defined over $\fq$. Let $f\in\ffq$ and suppose that $f\neq z^p-z$ for all $z\in\overline{\ffq}$.
Let $\psi$ be a non-trivial additive character of $\fq$. Then the bound
\[
\left|\sum_{Q\in\mathcal{H}} \psi \left( f(Q)\right) \right| \leq 2 \deg f q^{1/2}
\]
holds, where $\mathcal{H}$ is an arbitrary subgroup of $\ecq$.
\end{lemma}

\begin{cor}\label{cor:charSum}
Let $\E$ be an ordinary elliptic curve defined over $\fq$ and let  $\t_1,\dots, \t_s\in\End(\E)$ be pairwise non-associated elements with $(\t_i,u)=\End(\E)$, such that $0<\n(\t_1), \dots, \n(\t_s)\leq J$  and $\t_i/\t_j$ ($i\neq j$) are not powers of $\pi$. Let $f\in \ffq$ and suppose that $f\neq z^p-z$ for all $z\in\overline{\ffq}$.

For fixed elements $c_1\dots, c_s\in\fq$, not all zero, and a non-trivial additive character $\psi$   of $\fq$ we have
\[
 \sum_{Q\in\mathcal{H}}\psi\left( \sum_{i=1}^sc_if(\t_i Q) \right)\ll s J q^{1/2},
\]
where $\mathcal{H}$ is an arbitrary subgroup of $\ecq$.
\end{cor}

\subsection{Auxiliary results}

We start the following consequence of the sieve of Eratosthenes, see \cite{FriedlanderShparlinski_Power}.

\begin{lemma}\label{lemma:eratosthenes}
Let $\ell$ and $J$ be positive integers such that $J\geq \ell^{\varepsilon}$ with some  $\varepsilon>0$. Then
\[
  \sum_{\substack{0<a\leq J \\ \gcd(a,  \ell)=1}}1\gg \frac{J \varphi (\ell)}{\ell},
 \]
where the implied constant may depend on $\varepsilon$. 
\end{lemma}

We need the following estimation of the number of $\alpha \in \a$ with bounded norm  \cite{marcus}.

\begin{lemma} \label{lemma:paralelogramma}
Let $\O$ be an order of an imaginary quadratic field $K$ with discriminant $D$. For any $J>0$ and any ideal $\a\vartriangleleft \O$  we have
 \[
  \#\{\alpha\in\a: \n(\alpha)\leq J\}=C\cdot \frac{J}{\omega \sqrt{|D|}\n(\a)} +O\left(\frac{\sqrt{J}\delta(\a)+\delta(\a)^2}{\sqrt{|D|}\n(\a)}\right),
 \]
where $\omega=2,4,6$ is the number of units in $K$, $\delta(\a)$ is the longest diagonal of the fundamental parallelogram of $\a$ and $C$ is an absolute constant.
(Here $C=2\pi$, where here $\pi$ is the real number $\pi=3.14\dots$)
\end{lemma}

We give an analogue of Lemma \ref{lemma:eratosthenes} in the orders $\O$ of $K$.

\begin{lemma}\label{lemma:count-gcd=1}
Let $\O$ be an order of an imaginary quadratic field $K$ with conductor $u$ and discriminant $D$.
Let $J$ be positive integer and $\a$ be an ideal of $\O$ such that $\a$ is prime to $u$ and $J\geq \n(\a)^{\varepsilon}$ with some  $\varepsilon>0$. If $J\geq D^2$, then
\[
  \sum_{\substack{\alpha \in \O, \ \n(\alpha)\leq J\\ (\alpha)+\a=\O}}1\gg \frac{J}{\sqrt{|D|}}\cdot   \frac{\varphi_K(\a )}{\n(\a)},
 \]
where the implied constant may depend on $\varepsilon$. 
\end{lemma}

\begin{proof}
By Lemma \ref{lemma:UFD} we have
\begin{align*}
    \sum_{\substack{\alpha \in \O, \ \n(\alpha)\leq J\\ (\alpha)+\a=\O}}1=
  \sum_{  \b \mid \a } \mu_{K}(\b) \sum_{\substack{\alpha \in \b \\ \n(\alpha)\leq J}}1,
  \end{align*}
where we use the unique factorization of ideals $\b\mid \a $ follows from $(u)+\a=\O$.

We can assume, that $\b$ is square-free, otherwise  $\mu_K(\b)=0$.

In order to estimate the inner sum write $\b=\gamma \cdot \cc$ where $\gamma \in \O$ and $\cc$ has no principal divisor. Then  $\#\{\alpha \in \b: \n(\alpha)\leq J\}=\#\{\alpha \in \cc: \n(\alpha)\leq J/ \n(\gamma)\}$.  

Let   $(1,\omega)$ be the standard basis of $\O$ (i.e. $\omega=(D+\sqrt{D})/2$). Then by \cite[Proposition~5.2.1]{Cohen} $\cc$ has the form $\cc=(t,b+c\omega)$ with non-negative integers $0\leq b<t$, $c\mid t$, $\n(\cc)=tc$. As $\cc$ is square-free and principal-free, $c=1$. Then $\delta(\cc)\leq |t+b+\omega|\ll  \n(\cc)+|D|$.

If $\cc$ has norm $\n(\cc)\leq J/( \n(\gamma)  \n(\a)^{\varepsilon/2})$, we estimate the number of integers in $\cc$ by Lemma \ref{lemma:paralelogramma}:
\begin{align*}
 \sum_{\substack{\alpha \in \gamma\cc \\ \n(\alpha)\leq J}}1&= C\cdot \frac{J}{\omega \sqrt{|D|}\n(\gamma \cc)}\\ 
 & \quad +O\left(\frac{1}{\sqrt{|D|}}\left(
 \sqrt{\frac{J}{\n(\gamma)}} \left(1+\frac{|D|}{\n(\cc)}\right) + \n(\cc) + \frac{D^2}{\n(\cc)} \right)
 \right) \\
 &= C\cdot \frac{J}{\omega \sqrt{|D|}\n(\gamma \cc)} +O\left(\frac{1}{\sqrt{|D|}}
 \left(
 \sqrt{J}|D| + \frac{J}{  \n(\a)^{\varepsilon /2}} + D^2 \right)
 \right).
\end{align*}

Now assume that $\cc$ has norm $\n(\cc)\geq J/( \n(\gamma) \n(\a)^{\varepsilon/2})$, i.e. $\n(\b)\geq J/ \n(\a)^{\varepsilon/2}$. For a fixed imaginary part $\im (\alpha)$ there are at most $O(\sqrt{J}/\n(\b) )$ many integers $\alpha\in\b$ with $\n(\alpha)\leq J$, $\alpha\in\b$. As the height of the fundamental parallelogram is $\sqrt{|D|}/2$, there are at most $O(\sqrt{J}/\sqrt{|D|})$ different choices for the imaginary part. Thus we get the trivial bound
\begin{align*}
 \sum_{\substack{\alpha \in \b \\ \n(\alpha)\leq J}}1=O\left(\frac{J}{\sqrt{|D|}\n(\b)}\right)=O\left(\frac{\n(\a)^{\varepsilon/2}}{\sqrt{|D|}}\right)
\end{align*}

As for such $\a$ that $\n(\b)\geq J/\n(\a)^{\varepsilon/2}$ we also have
\[
 C\cdot \frac{J}{\omega \sqrt{|D|}\n(\b)}=O\left(\frac{\n(\a)^{\varepsilon/2}}{\sqrt{|D|}}\right),
\]
thus
\begin{align*}
    \sum_{\substack{\alpha \in \O, \ \n(\alpha)\leq J\\ (\alpha)+\a=\O}}1 &=  
   \sum_{\substack{\gamma, \cc : \, \gamma \cc \mid \a \\ \cc \text{ is principal-free}  }}
  \mu_{K}(\gamma \cc) \sum_{\substack{\alpha \in \cc \\ \n(\alpha)\leq J/\n(\gamma)}}1\\
   & =
   \frac{C}{\omega \sqrt{|D|}}
   \sum_{\substack{\gamma, \cc : \, \gamma \cc \mid \a \\ \cc \text{ is principal-free}  }}\mu_{K}(\gamma \cc)   
   \frac{J}{\n(\gamma \cc)} 
   +E
  \end{align*}
with error term
\begin{align*}
 E&\ll  \sum_{\gamma , \cc} 
 |\mu_{K}(\gamma \cc)|\, \frac{1}{\sqrt{|D|}} \left(
 |D| \sqrt{J} + \frac{J}{\n(\a)^{\varepsilon}} + D^2 +\n(\a)^{\varepsilon /2}
 \right) \\
 &\ll \frac{1}{\sqrt{|D|}} 2^{\omega_K(\a )}\left(
  |D| \sqrt{J}+ \frac{J}{\n(\a)^{\varepsilon}} + D^2 +\n(\a)^{\varepsilon /2}
 \right).
\end{align*}

The main term is
\begin{align*}
 &\frac{C}{\omega \sqrt{|D|}}J \sum_{\substack{\gamma, \cc : \, \gamma \cc \mid \a \\ \cc \text{ is principal-free}  }}  \frac{\mu_{K}(\gamma \cc) }{\n(\gamma \cc)} \\
 &=  \frac{C}{\omega \sqrt{|D|}}J  \sum_{\b \mid \a }   \frac{\mu_{K}( \a) }{\n( \a)}=\frac{C}{\omega \sqrt{|D|}}J  \frac{\varphi_K(\a )}{\n(\a )}.
\end{align*}
The main term dominates the error term, thus the result follows. 
\qed
\end{proof}

\begin{lemma}\label{lemma:rep}
Let $\O$ be an order of an imaginary quadratic field $K$ with conductor $u$ and discriminant $D$.
Let $\a$ be an ideal of $\O$ prime to $u$, $\t\in\O$ prime to $u$ and $\a$ and let $T$ be the multiplicative order of $\t$ modulo $\a$.

For a fixed $\rho\in\O$, prime to $\a$, let $M_\rho(J)$ be the number of pairs $(n,\gamma)$ with
 \begin{equation}\label{eq:rep}
     \t ^n \equiv \rho \cdot \gamma \mod \a, \quad 1\leq n\leq T, \ 0\leq \n(\gamma)\leq J.
 \end{equation}

For any fixed $\varepsilon>0$  if $J\geq  \n(\a)^{\varepsilon}$, there exists a $\rho\in\O$, prime to $\a$,  such that
\[
 M_\rho(J)\gg  \frac{ \sqrt{J} \, T}{ \ell^2 },
\]
where $\ell$ is the least positive integer in $\a$.

Moreover, if we also have $J\geq D^2$, then there exists a $\rho\in\O$, prime to $\a$, such that 
\[
  M_\rho(J)\gg \frac{1}{\sqrt{|D|}} \cdot \frac{ J\, T}{ \n(\a) }.
\]

The implied constants may depend on $\varepsilon$.
\end{lemma}

\begin{proof}
To prove the first part, we consider just rational solutions $\gamma$ of 
\[
  \t ^n \equiv \rho \cdot \gamma \mod \ell\O, \quad 1\leq n\leq T, \ 0\leq \n(\gamma)\leq J.
\]
As $\ell\in\a$, such a solution is also a solution of \eqref{eq:rep}. Moreover, as $\ell$ is the least positive integer in $\a$, all such solutions are also different modulo $\a$. Let $M^*_\rho(J)$ the number of solutions. As the pair $(n,\gamma)$ uniquely determines $\rho$, we get by Lemma \ref{lemma:eratosthenes}, that
\[
 \sum_{\substack{\rho \in\O/ \ell \O \\ (\rho,\ell)=\O}}M^*_{\rho}(J)\gg  T \sum_{\substack{\gamma\in\Z : \, \n(\gamma)\leq J\\ \gcd(\gamma,\ell)=1 }}1
  \gg T  \frac{\sqrt{J} \varphi(\ell)}{\ell}.
\]
Thus, there is a $\rho$, such that
\[
 M^*_{\rho}(J)\gg T \sqrt{J}  \frac{\varphi(\ell)}{\ell\,  \varphi_K(\ell)}= 
 \frac{T \sqrt{J}}{\ell^2}\frac{\displaystyle \prod_{\substack{t\mid \ell\\ t \text{ is prime}}} \left(1-\frac{1}{t}\right)}{\displaystyle \prod_{\substack{\p\mid \ell \\ \p \text{ is prime}}} \left(1-\frac{1}{\n(\p)}\right)}.
 \]

For those prime $t$ which is also prime in $K$, say $(t)=\p$, we have 
\[
\frac{1-t^{-1}}{1-\n(\p)^{-1}}=\frac{1-t^{-1}}{1-t^{-2}}=    \frac{1}{1+t^{-1}}\geq \frac{2}{3}.
\]
If $t=\p_1 \p_2$ (possibly $\p_1=\p_2$), then 
\[
\frac{1-t^{-1}}{(1-\n(\p_1)^{-1})(1-\n(\p_2)^{-1})}=\frac{1}{1-t^{-1}}\geq 1.
\]
As $M_\rho(J)\geq M^*_\rho(J)$, we get the result.

To prove the second part, observe that 
\[
 \sum_{\substack{\rho \in \O/\a \\ (\rho)+\a=\O}}M_{\rho}(J)\gg T \sum_{\substack{\gamma: \, \n(\gamma)\leq J\\ (\gamma)+\a=\O }}1
  \gg T \frac{J}{\sqrt{|D|}}\cdot  \frac{\varphi_K(\a)}{\n(\a)}
\]
by Lemma \ref{lemma:count-gcd=1}.
\qed
\end{proof}

\section{The elliptic curve endomorphism generator}

In this section we study the distribution and the linear complexity of sequences obtained from the point sequence $(\t^nP)$.

Clearly, the sequence $(\t^nP)$ is ultimately periodic. Let $\ell>1$ be the order of $P$, and let  $\l=\{\alpha \in \End (E): \alpha P=\infty\}$. By definition, the order $\ell$ is the least positive integer in $\l$, so $\sqrt{\n(\l)}\leq \ell\leq \n(\l)$.
If $\t$ is prime to $\l$, then $(\t^nP)$ is a purely periodic sequence, where the period length $T$ is the multiplicative order of $\t$  modulo $\l$. If $\l$ and $\t$ are prime to the conductor $u$ of $\End(\E)$, then the order $T$ is also the order of $\t$ in $\O_K/\l\O_K$ by Lemma~\ref{lemma:UFD}.

Consequently, $T\leq \n(\l)-1 \leq \ell^2-1$. If $\ell$ is a prime in $\O_K$, then $\l=(\ell)$ and $\O_K/\ell\O_K\cong\F_{\ell^2}$. Thus choosing $\t$ to be a generator of $\O_K/\ell\O_K$, one may have $T=\ell^2-1$.

\subsection{Distribution of the elliptic curve endomorphism generator}
In this section we study the distribution of the sequences $(f(\t^nP))$ for a rational function $f\in\ffq$. First we prove the following character sum estimate concerning the sequence $(f(\t^nP))$.

\begin{thm}\label{thm:charSum}
 Let $\E$ be an ordinary elliptic curve defined over $\fq$, let $u$ be the conductor of the endomorphism ring $\End(\E)$.
 Let $P\in\E(\F_{q})$ be of order $\ell\geq q^{1/4+\varepsilon}$ for some fixed $\varepsilon>0$, prime to $u$. Let $\t\in\End(\E)$ such that $\t$ is prime to $u$ and $\ell$.
 Let $T$ be the multiplicative order of $\t$ modulo $\l$. 
 
 Then for any rational function $f\in\ffq$ not having the form $z^p-z$ with $z\in\overline{\ffq}$,
 for any non-trivial additive character $\psi$ of $\fq$ and for any integer $\nu\geq 1$ the following bound holds:
 \[
  S_\t(\E,P,T)\ll \deg f \, T^{1-(3\nu+2)/2\nu(\nu+2)} \ell^{(2\nu+2)/\nu(\nu+2)} q^{1/4(\nu+2)} .
 \]
\end{thm}

\begin{proof}
Put
\[
 J=\left\lceil  \ell^{(4\nu+4)/(\nu+2)} q^{-1/(\nu+2)}T^{-2\nu/(\nu+2)}\right\rceil.
\]
Then $J\geq  \ell^{(4\nu+4)/(\nu+2)} q^{-1/(\nu+2)}T^{-2\nu/(\nu+2)}\geq \ell^{\varepsilon/(\nu+2)}$ so $J$ satisfies the condition of Lemma~\ref{lemma:rep}. Then  there is a $\rho$ prime to $\l$ such that the congruence \eqref{eq:rep} has $L\gg T \sqrt{J}/ \ell^2$ many solutions $(k_\gamma,\gamma)$, $\gamma\in\Gamma$. As $\End(\E)$ is an order of an imaginary quadratic field, the number of units are bounded (i.e. it is 2, 4 or 6), thus we can assume that the elements of $\Gamma$ are pairwise non-associated. Then
\begin{align*}
 S_\t(\E,P,T)&=  \sum_{s=1}^T \psi\left(f\left((\t^{s} P\right)\right)=
 \frac{1}{L}\sum_{\gamma\in\Gamma} \sum_{s=1}^T \psi\left(f\left(\t^{s+k_\gamma} P\right)\right)\\
 &=\frac{1}{L}\sum_{\gamma\in\Gamma} \sum_{s=1}^T \psi\left(f\left( \gamma\,\rho\,\t^{s} P\right)\right).
\end{align*}

By the  H\"older inequality we have
\begin{align*}
 |S_\t(\E,P,T)|^{2\nu}&\leq L^{-2\nu}T^{2\nu-1}\sum_{s=1}^T \left|\sum_{\gamma\in\Gamma} \psi\left(f\left( \gamma\,\rho\,\t^{s} P\right)\right) \right|   ^{2\nu}\\
 &\leq  L^{-2\nu}T^{2\nu-1}\sum_{Q\in \mathcal{H}} \left|\sum_{\gamma\in\Gamma} \psi\left(f\left( \gamma\,Q\right)\right) \right|   ^{2\nu},\\
 &=  L^{-2}    T \sum_{\gamma_1,\dots, \gamma_{2\nu}\in\Gamma} \sum_{Q\in \mathcal{H}} \psi\left(\sum_{i=1}^\nu \left(f(\gamma_i Q)-f(\gamma_{i+\nu}Q)\right)\right),
\end{align*}
where $\mathcal{H}$ is the group $\mathcal{H}=\{\alpha P: \alpha \in\End(\E)\}$. We remark, that $\mathcal{H} \leq E[\ell]$, thus $\#\mathcal{H} \mid \ell^2$. Moreover, as $P\in\E(\F_{q})$, i.e. $\pi P=P$, $\pi$ also fixes $\mathcal{H}$, thus $\mathcal{H} \leq \E(\F_{q})$.

If $(\gamma_1,\dots,\gamma_\nu)$ is a permutation of $(\gamma_{\nu+1},\dots,\gamma_{2\nu})$, we estimate the sum trivially by $\ell^2$.

Now assume, that  $(\gamma_1,\dots,\gamma_\nu)$ is not a permutation of $(\gamma_{\nu+1},\dots,\gamma_{2\nu})$.
If $\gamma\neq \gamma'$,  then $\gamma/\gamma'$  is not a power of $\pi$. Indeed, as $\pi$ fixes $P$, $\pi\equiv 1 \mod \l$. If $\gamma/\gamma'$ was a power of $\pi$, then we would have $\t^{k_\gamma-k_{\gamma'}}\equiv\gamma/\gamma' \equiv 1 \mod \l$, so $k_\gamma=k_{\gamma'}$ and $\gamma=\gamma'$. Then we can apply Corollary~\ref{cor:charSum}. 
As $\n(\gamma)\leq J$ for each $\gamma\in\Gamma$ these terms contribute at most $O(\deg f \, Jq^{1/2})$.
Thus
\begin{align*}
  |S_\t(\E,P,T)|^{2\nu}&\ll L^{-2\nu}    T^{2\nu-1}\left( L^\nu \ell^2 + L^{2\nu} \deg f \, Jq^{1/2}\right)\\
    &=     T^{2\nu-1}\left( L^{-\nu} \ell^2 + \deg f \, Jq^{1/2}\right)\\
     &\ll     T^{2\nu-1} \left( \left(\frac{T\sqrt{J}}{\ell^2 } \right)^{-\nu} \ell^2 + \deg f \, Jq^{1/2}\right)
\end{align*}
and so
\begin{align*}
 S_\t(\E,P,T) \ll \deg f \, T^{1-1/2\nu}\left( \left(\frac{T\sqrt{J}}{\ell^2 } \right)^{-1/2} \ell^{1/\nu} +  J^{1/2\nu}q^{1/4\nu} \right).
\end{align*}
Substituting the value of $J$ the result follows.
\qed
\end{proof}

If the discriminant $D_\E$ of the endomorphism ring $\End(\E)$ is small, we can give a stronger bound on $S_\t(\E,P,T)$.

\begin{thm}\label{thm:charSum-smallD}
 Having the same assumption as in Theorem \ref{thm:charSum} if $\n(\l)\geq q^{1/2+\varepsilon}$, then we have
  \begin{align*}
  S_\t(\E,P,T)\ll \max \Big\{ & \deg f \, |D_\E|^{1/4(\nu+1)}  T^{1-(2\nu+1)/2\nu(\nu+1)} \n(\l)^{1/2\nu} q^{1/4(\nu+1)} ,\\
  &\deg f \, |D_\E|^{1/\nu} T^{1-1/2\nu} q^{1/4\nu}\Big\}
 \end{align*}
 where $D_\E$ is the discriminant of the endomorphism ring $\End(\E)$.
\end{thm}

\begin{proof}
Put
\[
 J=\max\left\{ \left\lceil |D_\E|^{\nu/2(\nu+1)} \n(\l) q^{-1/2(\nu+1)} T^{-\nu/(\nu+1)}\right\rceil,D_\E^2 \right\}.
\]
Then $J\geq |D_\E|^{\nu/2(\nu+1)} \n(\l) q^{-1/2(\nu+1)} T^{-\nu/(\nu+1)}\geq  \n(\l)^{\varepsilon/ (\nu+1) }$ and trivially $J\geq D_\E^2$ so $J$ satisfies the condition of Lemma~\ref{lemma:rep}. Then  there is a $\rho$ prime to $\l$ such that the congruence \eqref{eq:rep} has $L\gg T J/ (\sqrt{|D_E|}\,\n(\l))$ many solutions. 

Using these solutions in the same way as in the proof of Theorem~\ref{thm:charSum}, we get
\begin{align*}
 &S_\t(\E,P,T) \ll\\
 &\deg f \, T^{1-1/2\nu}\left( 
 \left(
  \frac{1}{\sqrt{|D_\E|}} \frac{T J}{\n(\l)}
 \right)^{-1/2}\n(\l)^{1/2\nu}  
 +   J^{1/2\nu}q^{1/4\nu} \right).
\end{align*}
Substituting the value of $J$ the result follows.
\qed
\end{proof}

Theorem~\ref{thm:charSum} is non-trivial if $T\geq \ell^{4/3}q^{1/6+\varepsilon}$ and $\ell\geq q^{1/4+\varepsilon}$ by choosing $\nu$ to be large enough as
\[
 T^{-(3\nu+2)/2\nu(\nu+2)} \ell^{(2\nu+2)/\nu(\nu+2)} q^{1/4(\nu+2)} =\left(T^{-1}\ell^{\alpha_\nu} q^{\beta_\nu} \right)^{(3\nu+2)/2\nu(\nu+2)}
\]
where
\[
 \alpha_\nu=\frac{4}{3}\left(1+\frac{1}{3\nu+2}\right) \quad \text{and} \quad  \beta_\nu=\frac{1}{6}\left(1-\frac{2}{3\nu+2}\right).
\]

Similarly, Theorem~\ref{thm:charSum-smallD} is non-trivial if $T\geq \max\{ |D_\E|^{1/4}\n(\l)^{1/2}q^{1/4+\varepsilon}, \allowbreak D_\E^2 q\}$.

Theorem~\ref{thm:charSum-smallD} is sharper than Theorem~\ref{thm:charSum} for
$$
|D_\E|<\min\left\{\left(\frac{\ell^8}{T^2 \n(\l)^2}\right)^{1-\varepsilon}, \left(\frac{\ell ^2}{T}\right)^{1-\varepsilon}\right\}
$$  
if $\nu$ is large enough.

In the most interesting case, when $\l=(\ell)$, $T=\ell^{2+o(1)}$, Theorem~\ref{thm:charSum}, taken with $\nu=1$, yields
\[
  S_\t(\E,P,T) \ll T^{5/6+o(1)}q^{1/12},
\]
while Theorem~\ref{thm:charSum-smallD} yields
\[
  S_\t(\E,P,T) \ll \max\left\{ |D_\E|^{1/8} T^{3/4+o(1)}q^{1/8},  |D_\E| T^{1/2+o(1)}q^{1/4} \right\}.
\]

We apply Theorems \ref{thm:charSum} and \ref{thm:charSum-smallD} to investigate the distribution of the coordinates of $(f(\t^nP))$ in a fixed basis.

Namely, fix a basis $(\omega_1,\dots, \omega_{k})$ of $\F_{q}$ over its prime field $\fp$. For two integer vectors $(\alpha_1,\dots, \alpha_{k})$, $(\beta_1,\dots, \beta_{k})$ with $0\leq \alpha_i<\beta_i\leq p$, $i=1,\dots, k$, put
\[
 B_{[\alpha,\beta)}=\{\xi\in\F_{q}:\ \xi=\xi_1\omega_1+\dots +\xi_{k}\omega_{k}, \xi\in[\alpha_i,\beta_i), 1\leq 1\leq k \}
\]
of volume
\[
 \vol B_{[\alpha,\beta)}=\prod_{i=1}^{k}(\beta_i-\alpha_i)
\]
and denote by $N(\alpha,\beta)$ the number of sequence elements $(f(\t^nP))$ which hits the box $f(\t^nP)\in B_{[\alpha,\beta)}$, $1\leq n\leq T$. 

Let $\mathcal{B}$ denote the set of all such boxes, and $\Delta_{\t}(\E,P,T)$ denote the largest deviation of $N(\alpha,\beta)$ from its expected value
\[
 \Delta_{\t}(\E,P,T)=\sup_{B_{[\alpha,\beta)}\in\mathcal{B}}\left|N(\alpha,\beta)-\frac{\vol B_{[\alpha,\beta)}}{q}T \right|.
\]

Using the standard techniques to express the deviation  of $N(\alpha,\beta)$  from its expected value by character sums we get from Theorem~\ref{thm:charSum}.

\begin{cor}\label{cor:cor1} Having the same assumption as in Theorem \ref{thm:charSum} we have
 \begin{align*}
   &\Delta_{\t}(\E,P,T)\\   
   &\ll  \deg f \, T^{1-(3\nu+2)/2\nu(\nu+2)} \ell^{(2\nu+2)/\nu(\nu+2)} q^{1/4(\nu+2)} (\log p+1)^{k}.
\end{align*}
\end{cor}
\begin{proof}
 Let $\Psi$ be the set of additive characters of $\fq$. For any $\xi\in\fq$  we have
 \[
  \frac{1}{q}\sum_{\psi\in\Psi} \psi(\xi)=
  \left\{
  \begin{array}{cl}
   0 & \text{if } \xi=0,\\
   1 & \text{if } \xi\neq 0.
  \end{array}
\right.
  \]
Then
\begin{align*}
 N(\alpha,\beta)&=\frac{1}{q}\sum_{n=1}^T \sum_{\xi \in B_{[\alpha,\beta)}} \sum_{\psi\in\Psi} \psi(f(\t^nP)-\xi)\\
 &=\frac{1}{q}\sum_{\psi\in\Psi}\sum_{n=1}^T \psi(f(\t^nP))\sum_{\xi \in B_{[\alpha,\beta)}}\psi(-\xi).
\end{align*}
Separating the term corresponding to the trivial character $\psi_0$, we have
\[
 \left|N(\alpha,\beta)-\frac{\vol B_{[\alpha,\beta)}}{q}T \right|\leq \frac{1}{q}\sum_{\substack{\psi\in\Psi\\\psi\neq\psi_0}}\left|\sum_{n=1}^T \psi(f(\t^nP))\right|\left|\sum_{\xi \in B_{[\alpha,\beta)}}\psi(-\xi)\right|.
\]
We apply Theorem~\ref{thm:charSum} to the first term in the sum, and using the inequality
\[
 \frac{1}{q}\sum_{\psi\in\Psi}\left|\sum_{\xi \in B_{[\alpha,\beta)}}\psi(-\xi)\right|\leq (1+\log p)^k
\]
(see e.g. \cite{w2001}) we get the result.
\qed
\end{proof}

Applying  Theorem~\ref{thm:charSum-smallD} instead of Theorem~\ref{thm:charSum} in Corollary~\ref{cor:cor1} one gets
\begin{cor}
 Having the same assumption as in Theorem \ref{thm:charSum-smallD} we have
 \begin{align*}
   &\Delta_{\t}(\E,P,T)\\
   &\ll  \max \Big\{  \deg f \, |D_\E|^{1/4(\nu+1)}  T^{1-(2\nu+1)/2\nu(\nu+1)} \n(\l)^{1/2\nu} q^{1/4(\nu+1)}(\log p+1)^{k},\\
   &\qquad  \qquad  \deg f \, |D_\E|^{1/\nu} T^{1-1/2\nu} q^{1/4\nu}(\log p+1)^{k}\Big\}
\end{align*}
where $D_\E$ is the discriminant of the endomorphism ring $\End(\E)$.
\end{cor}

\subsection{Linear complexity of the elliptic curve endomorphism generator}

In this section we give a lower bound on the linear complexity $\mathcal{L}_\t(\E,P,T)$ of the sequence $(f(\t^nP))$.

\begin{thm}\label{thm:linComp}
Let $\E$ be an ordinary elliptic curve defined over $\fq$. Let $D_\E$ be the discriminant and $u$ be the conductor of the endomorphism ring $\End(\E)$.
Let $P\in\E(\F_{q})$ be of order $\ell$. Let $\t\in\End(\E)$ such that $\t$ is prime to $u$ and $\l$.
Let $T$ be the multiplicative order of $\t$ modulo $\l$.

Then for any rational function $f\in\ffq$ of degree $\deg f < \ell ^{2-\varepsilon}$ for some $\varepsilon>0$ which does not have the form $z^p-z$ with $z\in\overline{\ffq}$ the following bound holds.
\[
  \mathcal{L}_\t(\E,P,T) \gg  \max\left\{ \frac{T}{\ell^{4/3} (\deg f)^{1/3}} ,\frac{T}{ |D_\E|^{5/4} \n(\l)^{1/2}\, (\deg f)^{5/4}} \right\}.
 \]
\end{thm}

We need he following basic lemma about linear complexity \cite[Lemma~2]{shparlinski2001}.

\begin{lemma}\label{lemma:shp}
Let a sequence $(s_n)$ satisfy a linear recurrence relation 
\[
s_{n+L}=a_{L-1}s_{n+L-1}+\dots +a_1s_{n+1}+a_0s_n, \quad n=0,1,\dots
\] 
over $\fq$. Then for any $M\geq L+1$ pairwise distinct non-negative integers $j_1,\dots, j_M$ there exist $c_1,\dots, c_M \in \fq$, not all zero, such that
 \[
  \sum_{i=1}^Mc_is_{n+j_i}=0, \quad n=0,1,\dots
 \]
\end{lemma}

\begin{proof}[Proof of Theorem \ref{thm:linComp}]
Put $L=\mathcal{L}_\t(\E,P,T)$. To prove the first bound let 
$J=\lceil \ell^{4/3}/ (\deg f)^{2/3} \rceil$. 
By Lemma~\ref{lemma:rep} there exist $\rho$ and $M_{\rho}(J)\gg \sqrt{J}T/\ell^2 $ many values $k$ and $\gamma$ with \eqref{eq:rep}. As $\End(\E)$ is an order of an imaginary quadratic field, the number of units are bounded (i.e. it is 2, 4 or 6), thus we can assume that the elements $\gamma$ are pairwise non-associated.
 
If $L\geq M_{\rho}(J)$, then the result follows. Otherwise let us fix a $\rho$ and $L+1\leq M_{\rho}(J)$ such pairs $(k_i,\gamma_i)$ ($i=1,\dots,  L+1$) that satisfy \eqref{eq:rep}. By Lemma~\ref{lemma:shp} there exist $c_1,\dots, c_{L+1}\in\F_{q}$, not all zero, such that
 \[
  \sum_{i=1}^{L+1}c_i f\left(\t^{n+k_i}P \right)=0, \quad n=1,2,\dots,T.
 \]
Then 
 \[
  \sum_{i=1}^{L+1}c_i f\left(\t^{n+k_i}P \right)=\sum_{i=1}^{L+1}c_i f\left( \gamma_i\,\rho \,\t^{n}P \right)=0, \quad n=1,2,\dots,T.
 \]
Thus the function
\[
 F(Q)=\sum_{i=1}^{L+1}c_i f\left(\gamma_i Q \right)\in \F_{q}(\E)
\]
has at least $T$ many zeros. 

If $i\neq j$,  $\gamma_i/\gamma_j$  is not a power of $\pi$. Indeed, as $\pi$ fixes $P$, $\pi\equiv 1 \mod \l$. If $\gamma_i/\gamma_j$ was a power of $\pi$, then we would have $\t^{k_i-k_j}\equiv\gamma_i/\gamma_j \equiv 1 \bmod \l$, so $k_i=k_j$ and $i=j$. Then Lemma~\ref{lemma:ratFn-nonconstant} gives that $F$ is a non-constant function with degree at most $\deg F \leq (L+1) \,J \,\deg f$. Comparing the number of zeros and the degree we get the first bound.

To prove the second bound, as $\ell\leq \n(\l)$, 
we may assume, that $|D_\E|<\ell^{2/3} /(\deg f)^{1/2}$, otherwise the first bound is larger. Put  $$J=\lceil |D_\E|^{5/4} \n(\l)^{1/2} / (\deg f)^{1/4} \rceil.$$ As $J\geq |D_\E|^{5/4} \n(\l)^{1/2} / (\deg f)^{1/4}>D_\E^2$, by Lemma~\ref{lemma:rep}  there exist $\rho$ and $M_{\rho}(J)\gg JT/(\sqrt{|D_\E|}\ell^2)$ many values $k$ and $\gamma$ with \eqref{eq:rep}. If $L\geq M_{\rho}(J)$, then the result follows. Otherwise, we get $T\leq (L+1) \, J\, \deg f$ in the same way as before, and the result follows.
\qed
\end{proof}

\section*{Acknowledgments}
The author would like to thank Arne Winterhof and Igor Shparlinski for helpful comments. 

The author is partially supported by the Austrian Science Fund FWF Project F5511-N26 
which is part of the Special Research Program "Quasi-Monte Carlo Methods: Theory and Applications".

\end{document}